\documentclass[11pt, reqno]{amsart}

\usepackage{latexsym}
\usepackage{amsfonts}
\usepackage{amsmath}
\usepackage{url}
\usepackage{graphicx}
\usepackage{color}
\usepackage{bbm}
\definecolor{myblue}{rgb}{0.09,0.32,0.44} %22-84-113
\usepackage[pdfborder={0 0 0},pdfborderstyle={},colorlinks=true,linkcolor=myblue,citecolor=myblue,urlcolor=blue]{hyperref}

\usepackage{epsfig,psfrag,verbatim,amsfonts}

\usepackage{amssymb,amsmath}

\def\bl{\begin{lemma}}
\def\el{\end{lemma}}
\def\bth{\begin{theorem}}
\def\eth{\end{theorem}}
\def\bc{\begin{corollary}}
\def\ec{\end{corollary}}
\def\bcj{\begin{conjecture}}
\def\ecj{\end{conjecture}}
\def\bpr{\begin{proposition}}
\def\epr{\end{proposition}}
\def\bde{\begin{definition}}
\def\ede{\end{definition}}

\newcommand{\diam}{\mbox{\rm diam}}
\newcommand{\girth}{\mbox{\rm girth}}

\newcommand{\SL}{{\mathrm{SL}}}
\newcommand{\Cay}{{\mathrm{Cay}}}

\newtheorem{theorem}{Theorem}[section]
\newtheorem{definition}{Definition}[section]
\newtheorem{lemma}[theorem]{Lemma}

\newtheorem{corollary}[theorem]{Corollary}
\newtheorem{proposition}[theorem]{Proposition}
\newtheorem{conjecture}[theorem]{Conjecture}

\newtheorem*{theorem*}{Theorem}

\theoremstyle{definition}
\theoremstyle{remark}
\newtheorem{remark}[theorem]{Remark}
\numberwithin{equation}{section}

\begin{document}
\title{Expander spanning subgraphs with large girth}
\author{Itai Benjamini Mikolaj Fraczyk and G\'abor Kun}

\email{itai.benjamini@gmail.com}
\email{mfraczyk@math.uchicago.edu}
\email{kungabor@renyi.hu}

\date{20.12.30.}

\address{The Weizmann Institute, Rehovot, Israel}

\address{Department of Mathematics, University of Chicago, 5734 S. University Avenue, 
Chicago, IL, 60637, USA}

\address{Alfr\'ed R\'enyi Institute of Mathematics, H-1053 Budapest, Re\'altanoda u. 13-15., Hungary\\ Institute of Mathematics, E\"otv\"os L\'or\'and University, P\'azm\'any P\'eter s\'et\'any 1/c, H-1117 Budapest, Hungary}

\thanks{The first author thanks the Israeli Science Foundation for support. The third author is supported by the European Research Council (ERC) Advanced Grant No. 741420, by the J\'anos Bolyai Scholarship of the Hungarian Academy of Sciences and by the \'UNKP-20-5 New National Excellence Program of the Ministry of Innovation and Technology from the source of the National Research, Development and Innovation Fund.}

\begin{abstract}
We conjecture that finite graphs with positive Cheeger constant admit a spanning subgraph with positive Cheeger constant and girth proportional to the diameter. We prove this conjecture for regular expander graphs with large expansion. Our proof relies on the Lov\'asz Local Lemma. We find it quite surprising that it helps to get a large Cheeger constant.
\end{abstract}

\maketitle

\section{Introduction}

Let $G$ be a finite graph with vertex set $V(G)$. The Cheeger constant of $G$ is defined as \[h(G):=\inf_{\substack{S\subset V(G),\\ 0<|S|\leq |V(G)|/2}} \frac{|\partial_G S|}{|S|},\] where $\partial_G S$ denotes the edge boundary of $S$. We conjecture the following: 

\begin{conjecture}\label{mainConj} For every $\varepsilon>0$ and natural number $D$ there exist $\kappa, K>0$ such that if $G$ is a finite graph with maximum degree at most $D$ and
$h(G)>\varepsilon$ then $G$ admits a spanning subgraph $H$ with \[h(H) > \kappa\quad \textrm{ and }\quad\diam(H) <K \girth(H).\]
\end{conjecture}

% Originally you preferred this form of the conjecture.
%A finite graph $G$, is called $h$-{\it expander} for
%$$$ 
%h = \inf_{S \subset V_G :  \mbox{ } 0 < |S| < |G|/2} {|\partial S| \over |S|},
%$$
%where $V_{G}$ are the vertices of $G$ and $\partial S$ is the outer vertex boundary of $S$.
%$G$ a graph with degrees bounded from above by $d$ and expansion $h$.
%\begin{conjecture} There is a function $f(h, d) >0$ so that $G$ contains a spanning subgraph, which is 
%an $f(h,d)$-expander with girth proportional to the diameter. Where the proportionality ratio between girth and diameter,
%are bounded below by strictly positive functions only of the degrees and the expansion of $G$.
%\end{conjecture}

Our conjecture is a close relative of the Erd\H os-Hajnal and the Thomassen conjecture.
Erd\H os and Hajnal \cite{EH1,EH2} asked the following: 
does for every $k$ and $g$ there exist a $K=K(k,g)$ such that every finite graph with chromatic number 
at least $K$ contains a subgraph with chromatic number at least $k$ and girth at least $g$?
Thomassen \cite{T}  posed his conjecture in terms of the average degree:
does for every $d$ and $g$ there exist a $D=D(d,g)$ such that every finite graph with average degree 
at least $D$ contains a subgraph with average degree at least $d$ and girth at least $g$?
The case of regular graphs is handled in an unpublished work of Alon, see K\"uhn and Osthus \cite{KO}. 
However, the general case can not be reduced to this as proved by Pyber, R\"odl and Szemer\'edi \cite{PRS}. 
%K\"uhn and Osthus \cite{KO} settled the case $g=6$, while Dellamonica, Koubek, Martin and R\"odl \cite{DKMR} proved a directed version 
%of the conjecture. 

We can prove our conjecture in the case when $G$ is regular and it does not have too many short cycles. This will be a corollary of the following, more general theorem.

\begin{theorem}~\label{main}
Let $d, g, \delta$ be positive integers and let $G$ be a finite $d$-regular graph. Assume that for every $k \leq g$ every vertex $x$ is contained in at most $\big( \frac{d}{8\delta} \big)^k$ cycles of length $k$. Then $G$ admits a spanning subgraph $H$ with girth at least $g$ such that 
\[|\partial_H S| \geq \frac{\delta}{2d}|\partial_G S|-(2\log(d)+4)|S|\] holds  for every $S \subseteq V(G)$.
\end{theorem}

Note that the theorem is meaningless if $\delta<4\log(d)$ because then $H$ can be the edgeless spanning subgraph.

\begin{remark}
Our proof works beyond regular graphs. However, if $G$ is not regular then the ratio of the maximum and minimum degree appears in the theorem.
Note that $H$ may not be close to regular, not even when $G$ is. 
\end{remark}

\begin{corollary}~\label{Cheeger}
Let $d, g, \delta$ be positive integers and let $G$ be a finite $d$-regular graph. Assume that for every $k \leq g$ every vertex $x$ is contained in at most $\big( \frac{d}{8\delta} \big)^k$ cycles of length $k$ and that $h(G) \geq \frac{8d(\log(d)+2)}{\delta}$. Then $G$ admits a spanning subgraph $H$ with girth at least $g$ and $h(H) \geq \frac{\delta}{4d} h(G)$.
\end{corollary}

%%%%MF KILL THE m

We wish to combine the condition on the cycles and the Cheeger constant into one.
Recall that in a $d$-regular graph $G$ the Cheeger constant $h(G)$ and the second eigenvalue $\lambda_2(G)$ of the normalized combinatorial Laplacian are related by the inequality (see, e.g., \cite{Dod})
\[d\lambda_2(G)/2\leq h(G)\leq \sqrt{2d\lambda_2(G)}.\]

We can  deduce the following.
\begin{theorem}\label{cor-Spectral}
Let $d, g$ be positive integers. Let $G$ be a finite $d$-regular graph with $\lambda_2(G) \geq 1-\frac{1}{16(\log d+2)}$ and $|V(G)|>16(\log d+2)^g$. Then, the graph $G$ admits a spanning subgraph $H$ of girth at least $g$  with \[h(H) \geq \left(\frac{\lambda_2(G)(1-\lambda_2(G))^{-1}}{64}-2\log d -4\right).\]
In particular if $\lambda_2(G) \geq 1-\frac{1}{192(\log d+2)}$, then $h(H)\geq 0.99(\log d +2).$
\end{theorem}
%%%%MF
%%%%MF
\begin{remark} Ramanujan graphs of degree $d$  have $\lambda_2(G)\geq 1-\frac{2\sqrt{d-1}}{d}$ so Theorem \ref{cor-Spectral} is not vacuous.  It even gives a subgraph with expansion of order $\Omega(\sqrt d)$, on par with Ramanujan graphs. Random $d$-regular graphs are almost Ramanujan in the sense that $\lambda_2(G)= 1-\frac{2\sqrt{d-1}}{d}+o(d^{-1})$ with high probability \cite{Puder}, so by Theorem \ref{cor-Spectral} 
they admit spanning subgraphs of girth proportional to the diameter and expansion of order $\Omega(\sqrt d)$.\end{remark}
Our proof is an advanced version of the proof of Theorem 7 in \cite{K}, where the goal was to find a spanning subgraph $H$ with large girth, but there was no condition on the Cheeger constant of $H$.  The proof is based on the Local Lemma. We find it interesting that it helps to have a large Cheeger constant.

If the expansion is small, a natural approach to attack the conjecture is to work with a power graph, since that has larger spectral gap:
A spanning Lipschitz subgraph (where the endpoints of an edge are at most at a constant distance) $H$ was found in \cite{K} with large girth if $G$ is 
regular. %Passing to a spanning Lipschitz subgraph also helped finding a regular $H$: In fact a spanning $4$-Lipschitz subgraph suffices if the Cheeger constant of $G$ is large enough.
The core of this problem is that we do not know that if a power graph, say $G^2$ admits a spanning expander subgraph with large girth then so does $G$.

%Trimming short cycles (randomly) seems a natural approach, this has been applied successfully in percolation to get an acyclic giant component with  positive Cheeger constant,
%see \cite{abs, BLS, PS}.

Benjamini and Schramm \cite{BS} showed that infinite graphs with positive Cheeger constant contain a tree with positive Cheeger constant, solving a problem of Deuber, Simonovits and T. S\'os  \cite{DSS}. %This is an infinite variant on the conjecture for sufficiently large expansion. 

In \cite{K} a factor of iid version of this theorem was proved for Cayley graphs of non-amenable groups with large Cheeger constant. 
The measurable solution of Gaboriau and Lyons \cite{GL} to the von Neumann problem requires less: there is no bound on the length
of the edges in the spanning forest (but it should be a factor of iid) and it will be the Schreier graph of an almost free action of the free group. 
On the other hand, Gaboriau and Lyons showed that this action is ergodic. Their result had many applications to operator algebras, 
see Houdayer's Bourbaki seminar paper \cite{H}.

%%%%MF
Many interesting examples of expander families are provided by Cayley graphs of congruence quotients of atirhmetic groups. For example for any  $d\geq 2$ and a set $S \subset {\rm SL}_d(\mathbb Z)$ spanning a Zariski dense subgroup the family of Cayley graphs  ${\rm Cay}(\SL_d(\mathbb{F}_p), S)$ is an expander family with girth $\Omega (\log p)$. This is a rather deep fact originating in the work of Bourgain and Gamburd \cite{BoGa} and later expanded upon in \cite{BV}. For some of these arithmetic examples we can find evidence toward Conjecture \ref{mainConj} by hand. Let's take a closer look at $\SL_2(\mathbb Z).$ Let $S\subset \SL_2(\mathbb Z)$ be a finite symmetric set generating a Zariski dense subgroup. By the work of Bourgain and Gamburd \cite{BoGa} it is known that the sequence of graphs $G_p:=\Cay(\SL_2(\mathbb Z/p\mathbb Z), S)$ is an expander sequence as $p$ varies among sufficiently big primes. Let $\Gamma$ be the subgroup generated by $S$. The conjecture predicts that we should be able to find spanning subgraphs of $G_p$ of large girth which are still good expanders. In this example we can do something slightly weaker. We can find a sequence of expander graphs as spanning Lipschitz subgraphs $H_p$ of girth proportional to $\log p$. To find such graphs take a finite power of $S$ that contains two elements, say $a$ and $b$, that generate a free Zariski dense subgroup. This is always possible by the Tits alternative. Put $H_p:=\Cay(\SL_2(\mathbb Z/p\mathbb Z), \{a,b,a^{-1},b^{-1}\}).$ By \cite{BoGa} $H_p$ is an expander sequence. Since $\langle a,b\rangle$ is free, the girth must go to infinity. Indeed, any relation on $a,b$ that holds modulo infinitely many primes would have to hold in $\SL_2(\mathbb Z)$, which contradicts the freeness. In fact, it is easy to show that the girth of these graphs grows as $\Omega(\log p)$. 
%%%%

Arzhantseva and Biswas \cite{AB} constructed an expander sequence of Cayley graphs of $\SL_d(\mathbb{F}_p)$ with bounded diameter-girth ratio. Our theorem gives an alternative of their theorem, bounded diameter-girth ratio in a spanning subgraph of a good expander graph. 

\section{The Lov\'asz Local Lemma}

One of the most useful basic facts in probability is the following. If there is a finite set of mutually independent events that each of 
them holds with positive probability then the probability that all events hold simultaneously is still positive, although small. 
The Lov\'asz Local Lemma allows one to show that this statement still holds in case of rare dependencies.  

We will use the so-called {\bf variable version} of the lemma: We will consider a set of mutually independent 
random variables. Given an event $A$ determined by these variables 
we will denote by $vbl(A)$ the unique minimal set of variables that determines the event $A$: 
such a set clearly exists. Note that given the events $A, B_1, \dots ,B_m$ if $vbl(A) \cap vbl(B_i) = \emptyset$ 
for every $1 \leq i \leq m$ then $A$ is mutually independent of all the events $B_1, \dots ,B_m$.

\begin{lemma}\cite{EL}
Let $\mathcal{V}$ be a finite set of mutually independent random variables in a probability space.
Let $\mathcal{A}$ be a finite set of events determined by these variables. If there exists an assignment
$x: \mathcal{A} \rightarrow (0,1)$ such that for every $A \in \mathcal{A}$ we have
$\mathbb{P}(A) \leq x(A) \prod_{vbl(A) \cap vbl(B) \neq \emptyset} (1-x(B))$ then $\mathbb{P} \big( \bigwedge_{A \in \mathcal{A}} \overline{A} \big) \geq  \prod_{A \in \mathcal{A}} (1-x(A)).$
\end{lemma}

\section{The proof of Theorem~\ref{main}}

A large girth spanning subgraph was obtained under the condition on the number of cycles in \cite{K}, see Theorem 7. 
However, we also have an expansion condition on connected subsets.
Hence we study a probability distribution on spanning subgraphs where expansion is easier to analyse.
Surprisingly, the Local Lemma helps to achieve this goal, too.

First, we consider a digraph $D$ on the vertex set $V(G)$.
For every $(x,y) \in E(G)$ we add $(x,y)$ to $E(D)$ with probability $\frac{\delta}{d}$, and we add $(y,x)$ to $E(D)$ with probability $\frac{\delta}{d}$.
All these choices are independent. An undirected edge is present in $E(H)$ if at least one of the corresponding directed edges is in $E(D)$, this is,
$E(H)=\{(x,y): (x,y) \in E(D) \text{ or } (y,x) \in E(D)\}$.

We will apply the Local Lemma to this probability space. The set of variables $\mathcal{V}$ will correspond to the directed edges of $G$: 
Two directed edges correspond to every edge ($G$ has no loops).
We will call a cycle in $G$ {\it short} if it is shorter than $g$. The ''bad events'' of $\mathcal{A}$ correspond to either short cycles
or connected subsets of vertices. Let $\mathcal{C}$ denote the set of short cycles in $G$, and $\mathcal{S}$ the set of connected
subsets of vertices. For $C \in \mathcal{C}$ let $A_C$ denote the event that $E(C) \subseteq E(H)$.
Given a set $S \in \mathcal{S}$ let $A_S$ denote the event that $|\partial^{out}_D S| \leq \frac{\delta}{2d} |\partial_G (S)|-(2\log(d)+4)|S|$, where
$\partial^{out}_D S$ denotes the set of directed edges in $D$ with starting point in $S$ and endpoint not in $S$. Note that if this inequality holds for every connected set $S$ then it holds for every set.
To prove the theorem it suffices to show that with positive probability the set of events $\{ A_C: C \in \mathcal{C} \} \cup \{ A_S: S \in \mathcal{S}\}$ can be avoided. 

Since the choices for distinct edges are independent we get the following.

 {\bf Claim 1:} Let $C \in \mathcal{C}$ be a cycle of length $k$ in $G$. Then $\mathbb{P} (A_C) \leq (\frac{2\delta}{d})^k$.

{\bf Claim 2:} For every $S \in \mathcal{S}$ the inequality $\mathbb{P}(A_S) \leq (10d)^{-|S|}$ holds.

\begin{proof}
Consider the edges in $\partial_G (S)$. For every such edge we make a choice if we include it in $\partial^{out}_D S$ (with the right direction).
These are independent choices. We apply the Chernoff bound. 
Given a real number $\alpha$ consider 
\begin{align*}\mathbb{E} \alpha^{|\partial^{out}_D S|}=\sum_{k=0}^{|\partial_G (S)|} \binom{|\partial_G (S)|}{k} \big( \frac{\alpha\delta}{d}\big)^k  \big( \frac{d-\delta}{d}\big)^{|\partial_G (S)|-k}=\big( 1 + \frac{(\alpha-1) \delta}{d}\big)^{|\partial_G (S)|}.\end{align*}
Since $1 + \frac{(\alpha-1) \delta}{d} \leq \exp(\frac{(\alpha-1)\delta}{d} )$ the Markov inequality implies 
\[\mathbb{P}\Big(\alpha^{|\partial^{out}_D S|} \geq (10d)^{|S|} \exp \big( \frac{(\alpha-1)\delta}{d}|\partial_G (S)| \big) \Big) \leq (10d)^{-|S|}.\]
Choosing $\alpha=\frac{1}{2}$ we get that
\[\mathbb{P}\big(|\partial^{out}_D S| \leq \frac{-\log(10d)}{\log(2)}|S|  + \frac{\delta}{2\log(2)d}|\partial_G (S)| \big) \leq (10d)^{-|S|}\]
The claim follows.
\end{proof}

\begin{lemma}
Let $G$ be  a $d$-regular graph, $x$ a vertex of $G$ and $s$ a positive integer.
Then the number of connected, induced subgraphs of $G$ of size $s$ containing $x$ is at most
$d(d-1)^{s-2} {2s-2 \choose s-1} \leq (4d)^{s-1}.$
\end{lemma}

\begin{proof}
Given a connected, induced subgraph containing $x$ we explore it by running a Depth First Search algorithm starting at $x$.
The algorithm makes $(s-1)$ forward and $(s-1)$ backward steps: There are at most 
${2s-2 \choose s-1}$ ways to choose the order of these steps. We have at most $(d-1)$ choices when going forward,
but in the first step when we have $d$ possibilities. The lemma follows.
\end{proof}

We choose the function $x$ in order to apply the Local Lemma. For every $S \in \mathcal{S}$ set $x(A_S)=(8d)^{-|S|}$, and for every $C \in \mathcal{C}$
set $x(A_C)=\big(\frac{4\delta}{d}\Big)^{|C|}$, where $|C|$ denotes the length of the cycle $C$. We have to check the condition of the Local Lemma. 

First, we bound the product of terms of the form $(1-x(A))$. We will have separate bounds for short cycles and connected subsets.  
Let $v \in V(G)$. We use that the number of cycles of length $k$ containing $v$ is at most $\big( \frac{d}{8 \delta} \big)^k$.
\begin{align*}\prod_{v \in C \in \mathcal{C}} (1-x(A_C)) =&  \prod_{k=3}^g \prod_{v \in C \in \mathcal{C}, |C|=k} 
\Big(1-\big( \frac{4\delta}{d} \big)^k \Big)\geq \prod_{k=3}^g \Big(1-\big( \frac{4\delta}{d} \big)^k \Big)^{\big( \frac{d}{8 \delta} \big)^k} \\  \geq& \prod_{k=3}^{\infty} \exp(2^{-k}) = e^{-1/4}.\
 \end{align*}
Similarly,
\begin{align*}\prod_{v \in S \in \mathcal{S}} (1-x(A_S)) =& \prod_{k=1}^{\infty} \prod_{\substack{v \in S \in \mathcal{S}\\ |S|=k}}(1-x(A_S))
\geq\prod_{k=1}^{\infty}  \exp \big(-(8d)^{-k} (4d)^{k-1} \big)\\  \geq& \exp \big( \frac{1}{4d} \sum_{k=1}^{\infty} -2^{-k}\big) = e^{-1/4d}\end{align*}
We may assume that $d \geq 3$. Hence for every $S \in \mathcal{S}$
\begin{align*}\mathbb{P}(A_S) \leq& (10d)^{-|S|} \leq (8d)^{-|S|} \exp \big( -\frac{d+1}{4d} \big)^{|S|} \\ \leq& 
x(A_S) \Pi_{v \in S} \prod_{v \in C \in \mathcal{C}} (1-x(C)) \prod_{v \in T \in \mathcal{S}} (1-x(A_T))\end{align*} as needed.
Similarly, for every $C \in \mathcal{C}$
\begin{align*}\mathbb{P}(A_C) \leq& \big( \frac{2\delta}{d} \big)^{|C|} \leq \big(\frac{4\delta}{d}\Big)^{|C|} \exp \big( -\frac{d+1}{4d} \big)^{|C|}\\
 \leq& x(A_C) \prod_{v \in C} \prod_{v \in D \in \mathcal{C}} (1-x(A_D)) \prod_{v \in S \in \mathcal{S}} (1-x(S)).\end{align*} 

We can apply the Local Lemma and Theorem~\ref{main} follows.

\begin{remark} Our proof is based on the Local Lemma, hence it is not very efficient: It provides a(n exponentially) small lower bound 
on the probability that we find the required spanning subgraph. Recent approaches would give a polynomial time algorithm if the number
of events to avoid was polynomial. Unfortunately, this is not the case, since we may need to check every connected subset of the vertices.
However, in case when $G$ is an expander the Moser-Tardos algorithm \cite{MT} finds an expander $H$ in polynomial time.
\end{remark}

\section{The proof of Theorem \ref{cor-Spectral}}
%%%%MF
\begin{lemma}\label{lem-returns}
Let $G$ be a finite $d$-regular graph. Let $\lambda_2(G)$ be the second eigenvalue of the normalized Laplacian. Then for every vertex $v\in V(G)$ the number of cycles of length $k$ passing through $v$ is at most $d^k|V(G)|^{-1}+d^k(1-\lambda_2(G))^k$. 
\end{lemma}
\begin{proof}
Let $T\colon L^2(V(G))\to L^2(V(G))$ be the operator $T f (v)=\sum_{v'\sim v} f(v')$. It is the non-normalized transition operators of the random walk on $G$. The eigenvalues of $T$ are $d(1-\lambda_1),d(1-\lambda_2),\ldots, $ where $\lambda_i$ are the eigenvalues of the normalized combinatorial Laplacian. The number of closed cycles of length $k$ passing through $v$ is bounded by $\langle T^k \mathbf 1_v,\mathbf 1_v\rangle.$ Let $f_1,\ldots ,f_n$ be the normalized eigenvectors of $T$, with $f_1=|V(G)|^{-1/2}\mathbf 1_{V(G)}.$ We have $\mathbf 1_v=\sum_{i=1}^n \alpha_i f_i, $ with $\alpha_1=|V(G)|^{-1/2}$ and $\sum_{i=1}^n \alpha_i^2=1$. Hence 
\[ \langle T^k \mathbf 1_v,\mathbf 1_v\rangle= d^k|V(G)|^{-1}+ \sum_{i=2}^n |\alpha_i|^2 (d(1-\lambda_i))^k\leq d^k|V(G)|^{-1}+d^k(1-\lambda_2)^k.\]
%%%%
%Serre \cite{Serre} defined the operators $\Theta_r, r\in \mathbb N$ as 
%\[\Theta_r f(v)=\sum_{(v,v_1,\ldots,v_r)}f(v_r),\] where $(v,v_1,\ldots,v_r)$ is a non-backtracking path in $G$. It is clear the number of (pointed) cycles of length $k$ passing through $v$ is given by $\langle \Theta_k {\mathbf 1}_v,{\mathbf 1}_v\rangle.$ Serre gives a fomula for $\Theta_r$ in terms of $T$ via generating series:
%\[\sum_{r=0}^\infty \Theta_r x^r=\frac{1-x^2}{1-xT+(d-1)x^2}.\] 
%In particular, we deduce that for every $r$ there is a degree $r$ polynomial $P_{r,d}$ such that $\Theta_r=P_{r,d}(T)$.
\end{proof}

\begin{proof}[Proof of Theorem \ref{cor-Spectral}]
By Lemma \ref{lem-returns} each vertex $v\in V(G)$ is contained in at most $c_k:=d^k |V(G)|^{-1}+d^k(1-\lambda_2(G))^k$ cycles of length $k$. For $|V(G)|>16(\log d+2)^g$ we will have $c_k\leq 2 d^k(1-\lambda_2(G))^k$ for each $k\leq g$. Choose $\delta=(1-\lambda_2(G))^{-1}/16$. The cycle condition in Theorem \ref{main} is satisfied, so we get a spanning subgraph $H$ of girth at least $g$ such that 
\begin{align*}|\partial_H S| \geq& \frac{\delta}{2d}|\partial_G S|-(2\log(d)+4)|S|\geq \left(\frac{\delta}{4} \lambda_2(G)-(2\log d+4)\right)|S|\\
=& \left(\frac{\lambda_2(G)(1-\lambda_2(G))^{-1}}{64}-2\log d -4\right) |S|,
\end{align*} for every subset $S\subset V(G)$ with $|S|\leq |V(G)|/2$.
If $\lambda_2(G) \geq 1-\frac{1}{192(\log d+2)}$ then $(1-\lambda_2(G))^{-1}/64\geq 3(\log d+2)$. We deduce that in this case the Cheeger constant of $H$ satisfies $ h(H)\geq 0.99(\log d +4).$
\end{proof}

\section{Future directions}

It would be interesting to make the proof of Theorem~\ref{main} efficient in the general case, not only when $G$ is an expander graph.

Graph sparsification turned out to be a successful method in the last decade, see the work of Spielman and Srivastava \cite{SS}. %(started with Bencz\'ur and Karge). 
A Local Lemma based sparsification could be very interesting, especially if it provides a large girth 
graph in certain cases. Unfortunately, our proof does not achieve this goal, but a similar approach might work.

A strongly ergodic measurable solution to the von Neumann problem could probably have many applications, see \cite{GL} for details.
Theorem~\ref{cor-Spectral} implies a finite analogue of this.

Are there approachable high dimensional formulations of the conjecture for the several variants of high dimensional expanders \cite{L}?
For the standard definitions and background see \cite{HLW}. 

Let us end by recalling an old conjecture: there is no sequence of finite bounded degrees graphs growing  in size  to infinity, so that all the induced balls in all the graphs in the sequence are uniform expanders. 
For a related progress see \cite{FV}.

\textbf{Acknowledgment.}
The authors thank to G\'abor Pete for encouraging this collaboration and to Elad Tzalik for his suggestions to improve the paper.

\end{document}